\documentclass[11pt]{article}

\usepackage[T1]{fontenc}
\usepackage[utf8]{inputenc}
\usepackage{lmodern}
\usepackage{microtype}
\usepackage[a4paper,margin=1.05in]{geometry}
\usepackage{amsmath,amssymb,amsthm,mathtools}
\usepackage{mathrsfs}
\usepackage{array,booktabs}
\usepackage{enumitem}
\usepackage{xcolor}
\usepackage{comment}
\usepackage{hyperref}
\usepackage[nameinlink,noabbrev]{cleveref}

\usepackage{framed}
\usepackage{mdframed}
\newtheorem*{ftheo}{Main Theorem}
\newenvironment{fthm}
  {\begin{mdframed}[innertopmargin = 3pt, innerbottommargin=7pt,skipabove=5pt,skipbelow=5pt,linewidth=0.25pt,nobreak=true,align=center]\begin{ftheo}}
  {\end{ftheo}\end{mdframed}}

\usepackage{titlesec}
\titleformat{\subsection}[runin]{\normalfont\large\bfseries}{\thesubsection}{0em}{}

\hypersetup{
  colorlinks=true,
  linkcolor=blue!55!black,
  citecolor=green!35!black,
  urlcolor=blue!65!black,
  pdftitle={}
}

\setlist[itemize]{leftmargin=2em}
\setlist[enumerate]{leftmargin=2.25em}
\allowdisplaybreaks

\newtheorem{theorem}{Theorem}[section]
\newtheorem{proposition}[theorem]{Proposition}
\newtheorem{lemma}[theorem]{Lemma}
\newtheorem{corollary}[theorem]{Corollary}
\newtheorem{definition}[theorem]{Definition}

\newtheorem*{corollary*}{Corollary}

\newcommand{\kk}{\Bbbk}
\newcommand{\ZZ}{\mathbb Z}

\newcommand{\RR}{\mathbb R}
\newcommand{\La}{\Lambda}

\newcommand{\SA}{\mathcal A}
\newcommand{\ST}{\mathfrak t}
\newcommand{\FM}{\mathfrak M}
\newcommand{\FG}{\mathfrak G}
\newcommand{\lr}{\longrightarrow}
\newcommand{\Mat}{\operatorname{Mat}}
\newcommand{\Tor}{\operatorname{Tor}}
\newcommand{\diag}{\operatorname{diag}}
\newcommand{\id}{\mathrm{id}}

\newcommand{\cG}{\mathcal G}

\newcommand{\op}{\mathrm{op}}
\newcommand{\sse}{\subseteq}
\newcommand{\dd}{\operatorname{\partial}}
\newcommand{\st}{\operatorname{st}}
\newcommand{\fg}{\operatorname{fg}}
\newcommand{\dg}{\operatorname{dg}}
\newcommand{\proj}{\operatorname{proj}}
\newcommand{\perf}{\operatorname{Perf}}

\newcommand{\Mod}{\mathrm{Mod}}
\newcommand{\HH}{\operatorname{HH}}
\newcommand{\Der}{\operatorname{Der}}
\newcommand{\im}{\operatorname{im}}
\newcommand{\gr}{\operatorname{gr}}
\newcommand{\Perf}{\operatorname{Perf}}

\newcommand{\Om}{\Omega}
\newcommand{\natq}{\natural}

\newcommand{\DGA}{\operatorname{DGA}_{\kk}}

\title{{Vanishing of higher Legendrian homology for rainbow closures}}
\author{Roger Casals and Alexander Simons}
\date{}

\begin{document}
\maketitle

\begin{abstract}
We show that the higher Legendrian homology of any Legendrian link given by the rainbow closure of a positive braid must necessarily vanish. The result is proven in the most general setting: for any number of strands and crossings, integrally and with one basepoint per component of the Legendrian link.
\end{abstract}

\tableofcontents

%%%%%%%%%%%%%%%%%%%%%%%%%%%%%%%%%%%%%%%%%%%
%%%%%%%%%%%%%%%%%%%%%%%%%%%%%%%%%%%%%%%%%%%
%%%%%%%%%%%%%%%%%%%%%%%%%%%%%%%%%%%%%%%%%%%
\section{Introduction}\label{sec:intro}

The object of this note is to show that the higher Legendrian homology of any Legendrian link given by the rainbow closure of a positive braid must necessarily vanish. This is one of the first complete computations of the homology of a non-trivial Legendrian contact dg algebra for Legendrian links. Heretofore, the derivedness and the full non-commutativity of the Legendrian dg algebra have been the two key challenging aspects when trying to compute the higher homology of Legendrian dg algebras. We respectively address these two challenges by relating Legendrian dg algebras to derived Cohn localizations, and developing new results for these, including the study of non-commutative $LU$-decompositions in the derived setting. The note concludes with a computation of the Hochschild homology groups of the partially wrapped Fukaya categories associated to the max-tb Legendrian $(2,m)$-torus links, illustrating the strength of our results.

%%%%%%%%%%%%%%%%%%%%%%%%%%%%%%%%%%%%%%%%%%%

\subsection{ Scientific context.} Let $\La\sse(\RR^3,\xi_{\st})$ be a Legendrian link in the standard Darboux $(\RR^3,\xi_{\st})$. The study of such Legendrian links is a pillar of contact and symplectic topology in low dimensions, see e.g.~\cite{CasalsGao2022,CasalsGao2023,Chekanov2002,EliashbergPolterovich1996,EliashbergFraser2009,EtnyreHonda2005}. In particular, the Legendrian contact dg algebra, originally developed by Y.~Chekanov in \cite{Chekanov2002}, is a central Legendrian invariant for Legendrian links, see e.g.~ \cite{CasalsNg,Eliashberg2000,etnyre_ng_2020,Ng2008}. Many algebraic and geometric consequences can be extracted from the study of the Legendrian contact dg algebra, as illustrated by the wealth of results studying augmentations, generating families and rulings, $A_\infty$-categories and related structures, see e.g.~\cite{FuchsRutherford2011,FuchsIshkhanov2004,HenryRutherford2015,Leverson2016,Leverson2017,MelvinShrestha2005,NgRutherford2013,NgRutherfordShendeSivekZaslow2017,NgSabloff2006,Sabloff2005,Traynor2001}.\\

It has nevertheless proven difficult to study the Legendrian contact dg algebra in its entirety, and many of the above invariants are obtained by either linearizing the Legendrian contact dg algebra at a point, or studying geometric aspects of a certain subset of lower rank points. Due to its non-commutativity and the highly non-linear differential, the computation of the actual homology groups of many important Legendrian contact dg algebras has remained a challenge. The purpose of this article is to address this. Specifically, let $\beta\in\mbox{Br}_n^+$ be a positive braid word in $n$-strands and $\La_\beta\sse(\RR^3,\xi_{\st})$ the Legendrian link whose front is the rainbow closure of $\beta$, as depicted in \Cref{fig:FrontsBeta3}. This class of Legendrian links $\La_\beta$ contains an abundance of interesting examples, and includes Legendrian representatives of all torus knots, and infinitely many satellite and hyperbolic knots, cf.~\cite[Corollary 1.6]{CasalsGao2022} and \cite{CasalsNg}. In addition, any generic such Legendrian link admits infinitely many Lagrangian fillings.\\

The main contribution of this manuscript is the complete computation of the homology of the Legendrian contact dg algebras for such Legendrian links $\La_\beta\sse(\RR^3,\xi_{\st})$. We will consider the most general, and thus challenging, type of Legendrian contact dg algebra, which is also the most relevant at the geometric level: fully non-commutative, defined over $\ZZ$, and with one basepoint per component. This choice also makes our main result useful for the computation of partially wrapped Fukaya categories, see e.g.~\cite{Ganatraetal2023,Sylvan2016}, as illustrated in \Cref{sec:examples_applications}.

\begin{center}
	\begin{figure}[h!]
		\centering
		\includegraphics[scale=0.8]{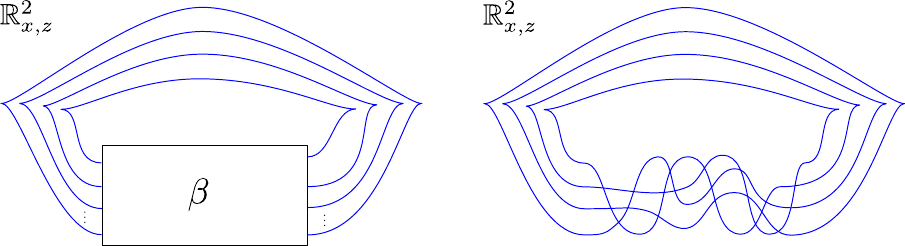}
		\caption{(Left) The rainbow closure of a positive braid word $\beta$, which is the front defining the Legendrian link $\La_\beta\sse(\RR^3,\xi_{\st})$. (Right) An instance for the 4-stranded positive braid word $\beta=\sigma_3\sigma_2(\sigma_1\sigma_2\sigma_3)^2\sigma_2\sigma_3\sigma_2\sigma_1\sigma_3\sigma_2\sigma_1\sigma_2\sigma_3\sigma_2\sigma_3\in\mbox{Br}_4^+$.}
		\label{fig:FrontsBeta3}
	\end{figure}
\end{center}

%%%%%%%%%%%%%%%%%%%%%%%%%%%%%%%%%%%%%%%%%%%
\subsection{ Main result.}  Let $\beta\in\mbox{Br}_n^+$ be a positive braid word and $\La_\beta\sse(\RR^3,\xi_{st})$ the Legendrian link whose front is the rainbow closure of $\beta$, as depicted in \Cref{fig:FrontsBeta3}. See also \cite[Section 2.2]{CasalsNg}. Let $\SA(\La_\beta)$ be the $\ZZ$-graded fully non-commutative Legendrian contact dg algebra over $\ZZ$ of the Legendrian link $\La_\beta$, where we choose exactly one basepoint per component of $\La_\beta$, cf.~\cite{etnyre_ng_2020} or \cite[Section 3.1]{CasalsNg}. The Maslov potential, which provides the $\ZZ$-grading, is chosen to be the standard Maslov potential, i.e.~each generator associated to a crossing of $\beta$ has degree 0 and each generator associated to a right cusp has degree 1, cf.~\cite[Section 5]{CasalsNg}. The main result we prove reads as follows:

\begin{fthm}\label{thm:main} Let $\beta\in\mbox{Br}_n^+$ be a positive braid word, $\La_\beta\sse(\RR^3,\xi_{st})$ its Legendrian link, and $(\SA(\La_\beta),\dd)$ its Legendrian contact dg algebra over $\ZZ$, with one basepoint per component. Then the higher homology of $(\SA(\La_\beta),\dd)$ vanishes, i.e.~
$$H_i(\SA(\La_\beta),\dd)=0,\quad \forall i\geq1.$$
\end{fthm}

The Main Theorem will be proven in two steps. First, we show in \Cref{ssec:Legendrian_DGA_derived_Cohn} that $(\SA(\La_\beta),\dd)$ is quasi-isomorphic to a certain derived Cohn localization. Second, we prove in \Cref{ssec:derived_Cohn_underived} that such derived Cohn localization is quasi-isomorphic to the underived universal localization, by leveraging the explicit hereditary properties present in the construction of $(\SA(\La_\beta),\dd)$. As a technical point, the Main Theorem will be established first over an arbitrary field $\kk$, and then we show that the corresponding higher vanishing also lifts integrally. The starting point for our argument is the description of $(\SA(\La_\beta),\dd)$ in our previous work \cite{CasalsNg}, see \Cref{ssec:dga_rainbow} below.\\

A direct consequence of the Main Theorem is:

\begin{corollary}\label{cor:main1}
The Legendrian dg algebra $(\SA(\La_\beta),\dd)$ of the rainbow closure of a positive braid word $\beta$ is quasi-isomorphic to the $\ZZ$-algebra $H_0(\SA(\La_\beta),\dd)$.
\end{corollary}

The dg category $\perf(\SA(\La_\beta))$ of perfect complexes over the dg algebra $(\SA(\La_\beta),\dd)$ is particularly important in the study of 4-dimensional symplectic topology. In particular, it is a model for the partially wrapped Fukaya category of the standard Darboux 4-ball $(\mathbb{D}^4,\omega_{\st})$ stopped at a ribbon of the Legendrian link $\La_\beta$, cf.~\cite{Ganatraetal2023,Sylvan2016}. To be precise, here we denote by $\perf(A)$ the idempotent-complete pretriangulated dg category of compact right dg $A$-modules. \Cref{cor:main1} implies the following:

\begin{corollary}\label{cor:main2}
Let $(\SA(\La_\beta),\dd)$ be the Legendrian dg algebra of the rainbow closure of a positive braid word $\beta$, and $A_\beta:=H_0(\SA(\La_\beta),\dd)$ its degree 0 homology algebra. Then there exists a dg quasi-equivalence
\begin{equation}\label{eq:derived_equiv_to_H0}
\perf(\SA(\La_\beta))\simeq C_{\dg}^b(\proj_{\fg}(A_\beta))
\end{equation}
\end{corollary}

Indeed, since the $0$th homology group $A_\beta:=H_0(\SA(\La_\beta),\dd)$ is an ordinary (underived) algebra, $\perf(H_0(\SA(\La_\beta)))$ is quasi-equivalent to $C_{dg}^b(\mbox{proj}_{fg}(A_\beta))$. Here $\mbox{proj}_{fg}(A_\beta)$ is the additive category of finitely generated projective (right) modules over $A_\beta$ and $C_{dg}^b$ the standard dg enhancement, whose objects are bounded complexes of finitely generated projective $A_\beta$-modules and whose morphism complexes are the usual Hom-complexes. At the level of homotopy categories \Cref{cor:main2} implies
\begin{equation*}\label{eq:derived_equiv_to_H0}
H^0(\perf(\SA(\La_\beta)))\simeq K^b(\proj_{\fg}(A_\beta)).
\end{equation*}

Thus we can study many key properties of $\perf(\SA(\La_\beta))$ by considering bounded complexes of finitely generated projective $A_\beta$-modules, which is significantly simpler. In particular, it follows from \Cref{cor:main2} that the derived invariants of $\perf(\SA(\La_\beta))$, such as Hochschild homology and its variants, or algebraic $K$-theory, can be computed using homological algebra for classical (underived) algebras and modules. In order to illustrate this, we compute in \Cref{sec:examples_applications} the entire Hochschild homology of the partially wrapped Fukaya category associated to the max-tb Legendrian $(2,m)$-torus link. For instance, we will show that the Main Theorem and \Cref{cor:main2} imply that all such Hochschild homology groups vanish in degrees above one.\\

%%%%%%%%%%%%%%%%%%%%%%%%%%%%%%%%%%%%%%%%%%%
%%%%%%%%%%%%%%%%%%%%%%%%%%%%%%%%%%%%%%%%%%%
\noindent {\bf Acknowledgements}. R.C.~thanks the hospitality of PCMI-IAS during the research program ``Knotted Surfaces in Four-Manifolds'', where parts of this article were developed. R.C.~is supported by the National Science Foundation under the grants DMS-1942363 and DMS-2505760, and a UC Davis College of L\&S Dean's Fellowship.\hfill$\Box$

%%%%%%%%%%%%%%%%%%%%%%%%%%%%%%%%%%%%%%%%%%%
%%%%%%%%%%%%%%%%%%%%%%%%%%%%%%%%%%%%%%%%%%%
\vspace{0.5cm}
\noindent {\bf Convention}. All rings are associative, unital, and homomorphisms preserve the
unit. dg algebras (DGAs) are homologically graded and their differentials have degree
\(-1\). Throughout the article $\kk$ is a field, and it is allowed to be of finite characteristic. All dg algebras will be considered over $\ZZ$ or a field $\kk$, and be $\ZZ$-graded. The natural number $n\in\mathbb{N}$ will always denote the number of strands of the positive braids and $m\in\mathbb{N}$ their length.\hfill$\Box$
%%%%%%%%%%%%%%%%%%%%%%%%%%%%%%%%%%%%%%%%%%%
%%%%%%%%%%%%%%%%%%%%%%%%%%%%%%%%%%%%%%%%%%%
%%%%%%%%%%%%%%%%%%%%%%%%%%%%%%%%%%%%%%%%%%%
\section{A preliminary lemma}\label{sec:preliminary}

%For a matrix $M=(m_{ij})\in\mbox{Mat}_\ell(S)$ with entries $m_{ij}$, we denote by $S\langle M\rangle$ the result of freely adjoining all the entries $m_{ij}$ to $S$, i.e.~$S\langle M\rangle:=S\langle m_{11},\ldots,m_{ij},\ldots m_{\ell \ell}\rangle$.

Let $\beta\in\mbox{Br}_n^+$ be a positive braid word, $\beta=\sigma_{i_1}\cdot\ldots\cdot\sigma_{i_m}$ in the positive Artin generators. Consider the Legendrian link $\La_\beta\sse(\RR^3,\xi_{st})$ whose front is the rainbow closure of $\beta$, as detailed in \cite[Section 2.2]{CasalsNg}. Depending on the basepoints, there are two Legendrian dg algebras that we consider for $\La_\beta$, as follows:

\begin{enumerate}[label=$(\roman*)$]
    \item The $\ZZ$-graded fully non-commutative Legendrian dg algebra $\SA(\La_\beta)$, where we choose exactly one basepoint per component of $\La_\beta$.

    \item The $\ZZ$-graded fully non-commutative Legendrian dg algebra $\SA(\La_\beta,\ST_{s})$, where we choose exactly one basepoint per strand of $\beta$, as in \Cref{fig:FrontsBeta2}(left).
\end{enumerate}

\noindent The Main Theorem is concerned with $\SA(\La_\beta)$ in $(i)$ above. That said, the proof for $\SA(\La_\beta)$ actually involves first proving the result for $\SA(\La_\beta,\ST_{s})$, as in $(ii)$ above. Note that both Legendrian dg algebras can be $\ZZ$-graded because the rotation class vanishes for any Legendrian link of the form $\La=\La_\beta$.\\

The Main Theorem is a statement about $\SA(\La_\beta)$ as an algebra over $\ZZ$. That said, the proof will first show that the statement holds for $\SA(\La_\beta,\ST_{s})$, and equivalently $\SA(\La_\beta)$, defined over a field $\kk$. Then we will argue that the vanishing also holds integrally, cf.~\Cref{ssec:proof_of_main_theorem}. Therefore, unless otherwise stated, in \Cref{sec:preliminary} and \Cref{sec:main_proof}, the symbols $\SA(\La_\beta)$ and $\SA(\La_\beta,\ST_{s})$ always denote the algebras defined over a fixed commutative field $\kk$, of any characteristic.

%%%%%%%%%%%%%%%%%%%%%%%%%%%%%%%%%%%%%%%%%%%
%%%%%%%%%%%%%%%%%%%%%%%%%%%%%%%%%%%%%%%%%%%
\subsection{ Legendrian dg algebra for rainbow closures.}\label{ssec:dga_rainbow} In this subsection we prove in \Cref{prop:dga-change} an explicit model for the quasi-isomorphism type of the Legendrian dg algebra of a rainbow closure which will be useful to prove the Main Theorem. Let \(\beta=\sigma_{i_1}\cdots\sigma_{i_m}\) be a positive \(n\)-stranded braid word. Following \cite[Section 5.1]{CasalsNg}, we introduce the following matrices:

\begin{definition}[Path matrices]\label{def:pathmatrix}
Let $n\in\mathbb{N}$, $i\in[1,n-1]\in\mathbb{N}$ and $z$ a variable over a ring $R$. The crossing matrix $P_i(z)\in \mbox{GL}(n,R\langle z\rangle)$ is defined as
$$
(P_i(z))_{jk} := \begin{cases} 1 & j=k \text{ and } j\neq i,i+1 \\
1 & (j,k) = (i,i+1) \text{ or } (i+1,i) \\
z & j=k=i+1 \\
0 & \text{otherwise;}
\end{cases},\qquad\mbox{i.e.}\quad
P_i(z):=\left(\begin{matrix}
1 & \cdots  & & & \cdots & 0\\
\vdots & \ddots & & & & \vdots\\
0 & \cdots & 0 & 1 & \cdots & 0\\
0 & \cdots & 1 & z & \cdots & 0\\
\vdots &  & & &\ddots & \vdots\\
0 & \cdots & & & \cdots & 1\\
\end{matrix}\right).
$$
Given a positive braid word $\beta=\sigma_{i_1}\cdots\sigma_{i_m} \in \mbox{Br}_{n}^{+}$  and $z_{1}, \dots, z_{m}$ free non-commuting indeterminate variables over $R$, we define the path matrix $P_{\beta}(z_1,\ldots,z_m)\in\mbox{GL}(n,R\langle z_1,\ldots,z_m\rangle)$ to be the product
$$
P_{\beta}(z_1,\ldots,z_m):=P_{i_1}(z_1)\cdots P_{i_m}(z_m).$$
\hfill$\Box$
\end{definition}

\noindent Note that the crossing matrix $P_i(z)$ is invertible, with two-sided inverse:
$$
(P_i^{-1}(z))_{jk} := \begin{cases} 1 & j=k \text{ and } j\neq i,i+1 \\
1 & (j,k) = (i,i+1) \text{ or } (i+1,i) \\
-z & j=k=i \\
0 & \text{otherwise;}
\end{cases},\mbox{ i.e.}\quad
P_i^{-1}(z):=\left(\begin{matrix}
1 & \cdots  & & & \cdots & 0\\
\vdots & \ddots & & & & \vdots\\
0 & \cdots & -z & 1 & \cdots & 0\\
0 & \cdots & 1 & 0 & \cdots & 0\\
\vdots &  & & &\ddots & \vdots\\
0 & \cdots & & & \cdots & 1\\
\end{matrix}\right),
$$
\noindent and thus the inverse of $P_{\beta}(z_1,\ldots,z_m)$ is $P^{-1}_{\beta}(z_1,\ldots,z_m)=P_{i_m}^{-1}(z_m)\cdots P_{i_1}^{-1}(z_1)$.

\begin{table}[htpb]
    \centering
    \begin{tabular}{lcc}
        \toprule
        \textbf{Generators} & \textbf{Degree} & \textbf{Index range} \\
        \midrule
        $z_{i}$ & 0 & $i\in [1,m]$ \\
        $d_i^{\pm1}$ & 0 & $i\in [1,n]$ \\
        $\ell_{ij}$ & 0 & $j<i$ \\
        $u_{ij}$ & $0$ & $i<j$ \\
        $b_{ij}$ & $1$ & $i,j\in [1,n]$ \\
        \bottomrule
    \end{tabular}
    \caption{List of generators for \(\cG(\beta)\) and their corresponding degrees. In rows three and four the range for $i,j\in[1,n]$ is constrained to $j<i$ and $i<j$ respectively. The only relations satisfied by the $d_i$-generators are $d_id_i^{-1}=d_i^{-1}d_i=1$.}
    \label{tab:generators}
\end{table}

The path matrices in \Cref{def:pathmatrix} can be used to describe $\SA(\La_\beta,\ST_{s})$ explicitly. Before that, let us introduce a second definition:

\begin{definition}\label{def:gauss-dga}
Let $\beta\in\mbox{Br}^+_n$ be a positive braid word in $n$-strands. By definition, the dg algebra \(\cG(\beta)\) is the free $\ZZ$-graded algebra over $\ZZ$ with graded generators as in \Cref{tab:generators} and differential
\begin{equation}\label{eq:gauss-diff}
 \partial B=P^{-1}_\beta(z_1,\ldots,z_m)+DLU,
 \end{equation}
 where the differential in \eqref{eq:gauss-diff} is extended by the graded Leibniz rule, and it is declared to vanish on any generators of degree 0. \Cref{eq:gauss-diff} is the compressed identity for the matrix equality
 {\tiny
 \begin{equation*}
 \begin{pmatrix}
\dd b_{11} & \dd b_{12} & \dots & \dd  b_{1n} \\
\dd b_{21} &\dd  b_{22} & \dots &\dd  b_{2n} \\
\vdots & \vdots & \ddots & \vdots \\
\dd b_{n1} &\dd  b_{n2} & \dots &\dd  b_{nn}
\end{pmatrix}
=
\begin{pmatrix}
p_{11} & p_{12} & \dots & p_{1n} \\
p_{21} & p_{22} & \dots & p_{2n} \\
\vdots & \vdots & \ddots & \vdots \\
p_{n1} & p_{n2} & \dots & p_{nn}
\end{pmatrix}
+
\begin{pmatrix}
d_{1} & 0 & \dots & 0 \\
0 & d_{2} & \dots & 0 \\
\vdots & \vdots & \ddots & \vdots \\
0 & 0 & \dots & d_{n}
\end{pmatrix}
\begin{pmatrix}
1 & 0 & \dots & 0 \\
\ell_{21} & 1 & \dots & 0 \\
\vdots & \vdots & \ddots & \vdots \\
\ell_{n1} & \ell_{n2} & \dots & 1
\end{pmatrix}
\begin{pmatrix}
1 & u_{12} & \dots & u_{1n} \\
0 & 1 & \dots & u_{2n} \\
\vdots & \vdots & \ddots & \vdots \\
0 & 0 & \dots & 1
\end{pmatrix}
 \end{equation*}
 }
\noindent where $p_{ij}$ is the $(i,j)$-entry of $P_\beta^{-1}$, $B:=(b_{ij})$ and
  \[
 D:=\diag(d_1,\ldots,d_n),\qquad
 L:=(\ell_{ij})\text{ lower unitriangular},\qquad
 U:=(u_{ij})\text{ upper unitriangular}.
\]

 \hfill$\Box$
 \end{definition}

Note that in \Cref{def:gauss-dga} we used {\it free graded algebra} but, technically, the generators $d_i$ are adjoined as being invertible. That is, we are adjoining $d_i$ and $d_i^{-1}$ with $d_id_i^{-1}=d_i^{-1}d_i=1$. Except for these, there are no relations among any of the generators in \Cref{tab:generators}, thus our use of the term {\it free} above.

\begin{center}
	\begin{figure}
		\centering
		\includegraphics[scale=0.8]{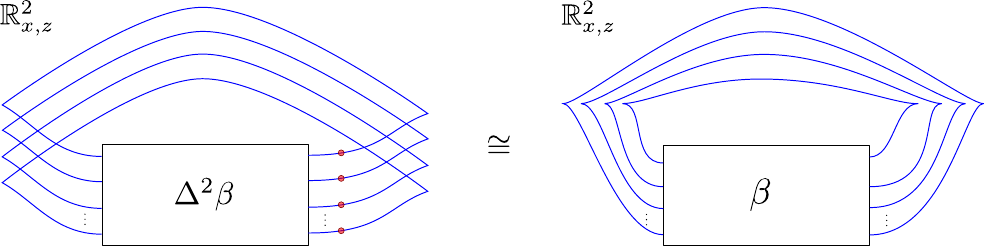}
		\caption{(Left) The $(-1)$-closure of $\Delta^2\beta$, where $\Delta$ is any positive braid word representing the half-twist in the same number of strands as $\beta$. The basepoints, one per strand, are drawn as red dots and located to the right of $\Delta^2\beta$, as depicted. (Right) The rainbow closure of $\beta$, which is Legendrian isotopic to the $(-1)$-closure of $\Delta^2\beta$. There is a Legendrian isotopy between the Legendrian link associated to the front on the left and the Legendrian link associated to the front on the right, see e.g.~\cite[Section 2]{CasalsNg}.}
		\label{fig:FrontsBeta2}
	\end{figure}
\end{center}

\begin{lemma}
\label{prop:dga-change}
Let $\beta\in\mbox{Br}_n^+$ be a positive braid word, $\La_\beta\sse(\RR^3,\xi_{st})$ its associated Legendrian link, and $\SA(\La_\beta,\ST_s)$ its Legendrian contact dg algebra over $\ZZ$, with one basepoint per strand. Then there exists a quasi-isomorphism
\begin{equation}
\SA(\La_\beta,\ST)\simeq\cG(\beta).
\end{equation}
\end{lemma}

\begin{proof}
First, let us perform a Legendrian isotopy to $\La_\beta$ so it becomes the $(-1)$-closure of the positive braid word $\Delta^2\beta$, as in \Cref{fig:FrontsBeta2}, cf.~\cite[Section 2.2]{CasalsNg}. By the stable-tame invariance of the Legendrian dg algebra, it suffices to show that the Legendrian dg algebra of the $(-1)$-closure of the positive braid word $\Delta^2\beta$ is quasi-isomorphic to $\cG(\beta)$. By \cite[Prop.~5.2]{CasalsNg}, the Legendrian dg algebra of such $(-1)$-closure is given by the free algebra with generators as in \Cref{tab:generators2}.

\begin{table}[htpb]
    \centering
    \begin{tabular}{lcc}
        \toprule
        \textbf{Generators} & \textbf{Degree} & \textbf{Index range} \\
        \midrule
        $z_{i}$ & 0 & $i\in [1,m]$ \\
        $t_i^{\pm1}$ & 0 & $i\in [1,n]$ \\
        $w_i$ & 0 & $i\in[1,n(n-1)]$ \\
        $c_{ij}$ & $1$ & $i,j\in [1,n]$ \\
        \bottomrule
    \end{tabular}
    \caption{List of generators for the Legendrian dg algebra of the $(-1)$-closure of the positive braid word $\Delta^2\beta$. The generators $w_i$ correspond to the crossings of $\Delta^2$, while the generators $z_i$ correspond to the crossings of $\beta$. The degree $1$ generators $c_{ij}$ are associated to the crossings in the pigtail closure of $\Delta^2\beta$, cf.~\cite[Section 2.2]{CasalsNg}.}
    \label{tab:generators2}
\end{table}
\noindent Its differential reads
\begin{equation}\label{eq:diff_dga_proof1}
\dd C=\mbox{Id}_n+P_{\Delta^2\beta}T
\end{equation}
where $C:=(c_{ij})$ records the degree 1 generators, and $T:=\diag(t_1,\ldots,t_n)$ the basepoint variables. By \cite[Lemma 2.3]{CGGS24}, there exists matrices $L=(\ell_{ij})$, $U=(u_{ij})$, which are lower and upper unitriangular respectively, such that $P_{\Delta^2}=LU$. Therefore these matrices also satisfy $P_{\Delta^2\beta}=LUP_{\beta}$ and the differential in \eqref{eq:diff_dga_proof1} reads as
\begin{equation}\label{eq:diff_dga_proof2}
\dd C=\mbox{Id}_n+LUP_{\beta}T
\end{equation}
Let us consider the matrix $B:=TCT^{-1}P_\beta^{-1}$. Then \eqref{eq:diff_dga_proof2} gives
\begin{align*}
 \partial B
 &=T(\partial C)T^{-1}P_\beta^{-1}\\
 &=T(\mbox{Id}_n+LUP_\beta T)T^{-1}P_\beta^{-1}\\
 &=P_\beta^{-1}+TLU,
 \end{align*}
which is the differential \eqref{eq:gauss-diff} for $\cG(\beta)$ after identifying $D=T$. In summary:
\begin{enumerate}[label=$(\roman*)$]
    \item We can express the generators $b_{ij}$ of $\cG(\beta)$ in terms of the generators in \Cref{tab:generators2} via the identity $B=TCT^{-1}P_\beta^{-1}$. The generators $z_i$ remain identical for all $i\in[1,m]$.

    \item Similarly, we can express the generators $\ell_{ij},u_{ij}$ of $\cG(\beta)$ in terms of the generators $w_i$ in \Cref{tab:generators2} via $P_{\Delta^2}=LU$, and declare $d_i=t_i$ for all $i\in[1,n]$.

    \item These identifications above, along with the Legendrian isotopy, define a stable-tame equivalence of graded algebras. The computation $\dd B=P_\beta^{-1}+TLU$ above shows that it induces a quasi-isomorphism of dg algebras.
\end{enumerate}
Thus the Legendrian dg algebra of the $(-1)$-closure of the positive braid word $\Delta^2\beta$ is quasi-isomorphic, as a dg algebra, to the dg algebra $\cG(\beta)$. Therefore our original Legendrian dg algebra $\SA(\La_\beta,\ST)$ is quasi-isomorphic to $\cG(\beta)$, as required.
\end{proof}

\noindent From now on, we fix a field $\kk$ and consider algebras over $\kk$ unless otherwise stated. To ease notation, we will often still denote by $\cG(\beta)$ the base change of the $\ZZ$-algebra from \Cref{def:gauss-dga} to such a field $\kk$.

%%%%%%%%%%%%%%%%%%%%%%%%%%%%%%%%%%%%%%%%%%%
%%%%%%%%%%%%%%%%%%%%%%%%%%%%%%%%%%%%%%%%%%%

\subsection{ Cohn localization and non-commutative $LU$-decompositions}\label{ssec:Cohn_loc_Gauss_dec}
Based on \Cref{prop:dga-change}, we shift our focus to studying the dg algebra $\cG(\beta)$ in \Cref{def:gauss-dga}. An important observation now is that \eqref{eq:gauss-diff} can be interpreted as $P_\beta^{-1}(z_1,\ldots,z_m)$ admitting an LU-decomposition up to homotopy, i.e.~in the derived sense. Thus, as a first underived step, we now address the question of characterizing matrices with entries in $\kk\langle z_1,\ldots,z_m\rangle$ that admit such an LU-decomposition:

\begin{lemma}\label{lem:LU_decomposition}
Let \(S\) be a unital, not necessarily commutative, ring, and let
\(Y\in\Mat_n(S)\). Then the following are equivalent:
\begin{enumerate}[label=\textup{(\roman*)},leftmargin=2.2em]
\item Every leading principal submatrix of $Y$ is invertible.
\item There is a factorization $Y=D_YL_YU_Y$, where \(D_Y\) is invertible diagonal, \(L_Y\) is lower
unitriangular, and \(U_Y\) is upper unitriangular.
\end{enumerate}
In addition, if such factorization exists, it must be unique.
\end{lemma}

\begin{proof}
To ease notation we denote by $ Y^{[r]}$ the $r\times r$ leading principal submatrix of a matrix $Y$. Let us first show $(ii)\Longrightarrow(i)$, so we assume that the factorization \(Y=D_YL_YU_Y\) exists.  Since the factors are respectively
diagonal, lower triangular, and upper triangular, taking a leading
principal block commutes with their product. There we obtain
\[
 Y^{[r]}=D_Y^{[r]}L_Y^{[r]}U_Y^{[r]}.
\]
Every factor on the right is invertible, so \(Y^{[r]}\) is invertible, and $(ii)\Longrightarrow(i)$ is proven.\\

For $(i)\Longrightarrow(ii)$, we suppose that all the leading principal submatrices \(Y^{[r]}\) are invertible.  We argue by
induction on the size \(n\in\mathbb{N}\).  The case \(n=1\) holds tautologically. Let us write
\[
 Y=\begin{pmatrix}A&b\\ c&x\end{pmatrix},
\]
where \(A=Y^{[n-1]}\).  By the induction hypothesis, there is a unique factorization
\[
 A=D'L'U'.
\]
By setting $s:=x-cA^{-1}b$, note that the matrices
\[
 P:=\begin{pmatrix}I&0\\-cA^{-1}&1\end{pmatrix},
 \qquad
 Q:=\begin{pmatrix}I&-A^{-1}b\\0&1\end{pmatrix}
\]
are invertible and satisfy
\[
 PYQ=\begin{pmatrix}A&0\\0&s\end{pmatrix}.
\]
Since \(P,Y,Q\) are invertible, the block-diagonal matrix
\(\diag(A,s)\) is invertible.  Multiplying by
\(\diag(A^{-1},1)\) shows that \(\diag(\mbox{Id},s)\) is invertible. The
bottom-right equations for a two-sided inverse then give a two-sided
inverse for \(s\), and thus \(s\) is invertible.
Set
\[
 d_n:=s,
 \qquad
 u:=(D'L')^{-1}b,
 \qquad
 \ell:=s^{-1}c(U')^{-1},
\]
and define the following extensions of $D',L'$ and $U'$:
\[
 D_Y:=\begin{pmatrix}D'&0\\0&s\end{pmatrix},
 \quad
 L_Y:=\begin{pmatrix}L'&0\\\ell&1\end{pmatrix},
 \quad
 U_Y:=\begin{pmatrix}U'&u\\0&1\end{pmatrix}.
\]
Then the upper-left, upper-right, and lower-left blocks of
\(D_YL_YU_Y\) are respectively \(A,b,c\).  Its lower-right entry is
\begin{align*}
 s(1+\ell u)
 &=s+c(U')^{-1}(D'L')^{-1}b\\
 &=s+cA^{-1}b=x,
\end{align*}
where
\((U')^{-1}(D'L')^{-1}=A^{-1}\).  Thus \(Y=D_YL_YU_Y\), which proves  $(i)\Longrightarrow(ii)$.\\

\noindent For uniqueness, we suppose that \(Y=D_YL_YU_Y\) and proceed by induction, as above.  The induction hypothesis is that leading
\((n-1)\times(n-1)\) block determines \(D',L',U'\).  The
upper-right block then determines
\[
 u=(D'L')^{-1}b.
\]
Using the lower-left and lower-right submatrices we obtain
\[
 x=d_n+c(U')^{-1}u,
\]
so
\[
 d_n=x-c(U')^{-1}u=x-cA^{-1}b=s.
\]
and we also have $\ell=d_n^{-1}c(U')^{-1}$. Thus every factor is uniquely determined.
\end{proof}

\Cref{lem:LU_decomposition} allows us to describe the zeroth homology group $H_0(\cG(\beta))$ of the dg algebra $\cG(\beta)$ defined in \Cref{def:gauss-dga}. Indeed, by \Cref{lem:LU_decomposition}, $H_0(\cG(\beta))$ is the (underived) Cohn localization of $\kk\langle z_1,\ldots,z_m\rangle$ at the leading principal submatrices of $P_\beta^{-1}(z_1,\ldots,z_m)$. More precisely, we will now present an explicit model to describe such Cohn localization, which will be rigorously justified via \Cref{lem:LU_decomposition}, as follows.

\begin{definition}\label{def:FM_and_FG} Let us consider the following two algebras and a morphism between them:

\begin{enumerate}[label=$(\roman*)$]
    \item $\FM:=\kk\langle x_{ij}\rangle$ is the unital associative free algebra on $n^2$ generators $x_{ij}$, $i,j\in[1,n]$.

    \item  $\FG:=
 \kk\langle d_i^{\pm1},\ell_{ij},u_{ij}: d_id_i^{-1}=1=d_i^{-1}d_i\rangle$, is the unital associative free algebra, non-commutative Laurent in the $d_i$ generators, and where the indices range in $j<i$ for $\ell_{ij}$, $i<j$ for $u_{ij}$, and in general $i,j\in[1,n]$.

 \item The algebra morphism $\gamma:\FM\lr\FG$ is defined by
 $$\gamma(x_{ij})=(i,j)\mbox{-entry of }DLU,$$
 where, as in \Cref{ssec:dga_rainbow} above, we denote
 $$ D:=\diag(d_1,\ldots,d_n),\quad
 L:=(\ell_{ij})\text{ lower unitriangular},\quad
 U:=(u_{ij})\text{ upper unitriangular}.$$
\end{enumerate}
\hfill$\Box$
\end{definition}

Endowed with \Cref{def:FM_and_FG}, \Cref{lem:LU_decomposition} now implies the following result:

\begin{corollary}
\label{cor:gamma-localization}
The map \(\gamma:\FM\lr\FG\) is the (underived) Cohn localization of
\(\FM\) at the collection of the leading principal submatrices of $X=(x_{ij})$. That is, for every unital ring \(S\), composition with \(\gamma\)
identifies ring maps \(\FG\to S\) with ring maps \(f:\FM\to S\)
for which every image of a leading principal submatrix of $X$ under \(f\) is invertible.
\end{corollary}

\begin{proof}
Indeed, a map \(\FG\to S\) is exactly the choice of an invertible diagonal
matrix \(D_S\), a lower unitriangular matrix \(L_S\), and an upper
unitriangular matrix \(U_S\).  Its restriction to \(\FM\) sends \(X\)
to \(D_SL_SU_S\), whose leading blocks are invertible by \Cref{lem:LU_decomposition}.
Conversely, if the leading blocks of \(f(X)\) are invertible, \Cref{lem:LU_decomposition}
provides a unique factorization \(f(X)=D_SL_SU_S\), hence a unique
extension \(\FG\to S\).  This is precisely the universal property
of Cohn localization, cf.~\cite[Chapter 7]{CohnBook}.
\end{proof}

In order to explicitly describe derived Cohn localization, as used in \Cref{sec:main_proof}, it is useful to functorially express the underived Cohn localization as a pushout, using the algebras and map from \Cref{def:FM_and_FG}:

\begin{corollary}
\label{cor:strict-base-change}
Let \(A\) be a dg algebra under \(\FM\), via $g:\FM\lr A$, and suppose the entries $g(x_{ij})$ are degree-zero cycles in $A$.  Then the (underived) pushout $ \FG*_{\FM}A$ is the (underived) Cohn localization of \(A\) at the images of the leading principal submatrices of $(x_{ij})$ under the map $g$.
\end{corollary}

\begin{proof}
For any target dg algebra \(S\), the universal property of the pushout implies that a dg algebra map \(\FG*_{\FM}A\lr S\) is a dg algebra map \(A\lr S\)
together with an extension of its restriction \(\FM	\lr S\) across
\(\FG\).  By \Cref{cor:gamma-localization}, such an extension
exists (and it is unique) precisely when all the indicated leading principal submatrices
are strictly invertible in \(\Mat(S)\), which is the universal property of (underived) Cohn
localization.
\end{proof}
\color{black}
In \Cref{sec:main_proof} we will use a derived version of \Cref{cor:strict-base-change}, modeling derived Cohn localization via the corresponding derived pushout.
%%%%%%%%%%%%%%%%%%%%%%%%%%%%%%%%%%%%%%%%%%%
%%%%%%%%%%%%%%%%%%%%%%%%%%%%%%%%%%%%%%%%%%%
%%%%%%%%%%%%%%%%%%%%%%%%%%%%%%%%%%%%%%%%%%%
\section{Proof of main result}\label{sec:main_proof}

We structure the proof of the Main Theorem by building the following sequence of quasi-isomorphisms:

\begin{equation}\label{eq:sequence_qiso}
\SA(\La_\beta,\ST)\stackrel{(a)}{\simeq}\cG(\beta)\stackrel{(b)}{\simeq}L^{\mathrm{dga}}_{\Sigma_\beta}(\kk\langle z_1,\ldots,z_m\rangle)\stackrel{(c)}{\simeq} \kk\langle z_1,\ldots,z_m\rangle_{\Sigma_\beta}.
\end{equation}

Let us first present the ingredients featuring in \eqref{eq:sequence_qiso}:

\begin{enumerate}[label=(\alph*)]
    \item $\SA(\La_\beta,\ST)$ is the Legendrian dg algebra of the rainbow closure $\Lambda_\beta$ of the positive braid word $\beta$ with one basepoint per strand. The dg algebra $\cG(\beta)$ is defined in \Cref{def:gauss-dga}. We established the quasi-isomorphism $(a)$ in \eqref{eq:sequence_qiso} in \Cref{prop:dga-change} above.

    \item $\Sigma_\beta$ is the finite collection of matrices over $\kk\langle z_1,\ldots,z_m\rangle$ given by the leading principal submatrices of the path matrix $P_\beta^{-1}$, and $L^{\mathrm{dga}}_{\Sigma_\beta}(\kk\langle z_1,\ldots,z_m\rangle)$ the derived Cohn localization of $\kk\langle z_1,\ldots,z_m\rangle$ along $\Sigma_\beta$. See ~\cite[Definition 4.35]{BCL} for details on the derived Cohn localization. The quasi-isomorphism $(b)$ in \eqref{eq:sequence_qiso} will be proven in \Cref{prop:derived-model} below.

    \item $\kk\langle z_1,\ldots,z_m\rangle_{\Sigma_\beta}$ denotes the underived Cohn localization of $\kk\langle z_1,\ldots,z_m\rangle$ along $\Sigma_\beta$, as constructed in \cite[Chapter 7]{CohnBook}. The quasi-isomorphism $(c)$ in \eqref{eq:sequence_qiso} will be proven in \Cref{prop:ordinary-equals-derived-localization} below.
\end{enumerate}

Since the underived Cohn localization $\kk\langle z_1,\ldots,z_m\rangle_{\Sigma_\beta}$ is an underived algebra, i.e.~ a dg algebra concentrated in degree 0 and trivial differential, its higher homology vanishes. Thus once we establish the quasi-isomorphisms in \eqref{eq:sequence_qiso}, it will follow that the same holds for $\SA(\La_\beta,\ST)$, as required. The focus of the rest of this subsection is on proving \Cref{prop:derived-model} and \Cref{prop:ordinary-equals-derived-localization}, which respectively establish the quasi-isomorphisms $(b)$ and $(c)$ in \eqref{eq:sequence_qiso}.

%%%%%%%%%%%%%%%%%%%%%%%%%%%%%%%%%%%%%%%%%%%
%%%%%%%%%%%%%%%%%%%%%%%%%%%%%%%%%%%%%%%%%%%

%%%%%%%%%%%%%%%%%%%%%%%%%%%%%%%%%%%%%%%%%%%
%%%%%%%%%%%%%%%%%%%%%%%%%%%%%%%%%%%%%%%%%%%
\subsection{ Legendrian dg algebra as a derived Cohn localization.}\label{ssec:Legendrian_DGA_derived_Cohn} The goal of this subsection is to show that the Legendrian dg algebra of a rainbow closure $\La_\beta$ is quasi-isomorphic to a derived Cohn localization of $\kk\langle z_1,\ldots,z_m\rangle$ along certain submatrices of the inverse path matrix $P_\beta^{-1}$. Specifically, we will now prove the following result:

\begin{theorem}
\label{prop:derived-model}
Let $\beta\in\mbox{Br}_n^+$ be a positive braid word and $\Sigma_\beta$ the collection of leading principal submatrices of the inverse path matrix $P_\beta^{-1}$. Then there exists a quasi-isomorphism of dg algebras:
\begin{equation}\label{eq:LegendrianDGA_derived_Cohn}
 \cG(\beta)\simeq L^{\mathrm{dga}}_{\Sigma_\beta}(\kk\langle z_1,\ldots,z_m\rangle)   
\end{equation}
\end{theorem}

\noindent Before delving into the proof of \Cref{prop:derived-model}, we must discuss an appropriately explicit model for the derived Cohn localization in the right-hand side of \eqref{eq:LegendrianDGA_derived_Cohn}. Note that the left hand side of \eqref{eq:LegendrianDGA_derived_Cohn} is already explicitly understood, cf.~\Cref{def:gauss-dga}.

\subsubsection{A resolution to compute the derived Cohn localization.}\label{sssec:model_resolution_derived_Cohn} By \Cref{cor:strict-base-change} and \cite[Theorem~4.36]{BCL}, the derived Cohn localization $L^{\mathrm{dga}}_{\Sigma_\beta}(\kk\langle z_1,\ldots,z_m\rangle)$ of $\kk\langle z_1,\ldots,z_m\rangle$ along the principal submatrices in $\Sigma_\beta$ can be computed as the derived pushout
\begin{equation}\label{eq:derived_pushout}
L^{\mathrm{dga}}_{\Sigma_\beta}(\kk\langle z_1,\ldots,z_m\rangle)\simeq \FG*^{\mathbf L}_{\FM}\kk\langle z_1,\ldots,z_m\rangle
\end{equation}
where the map $\FM\lr\kk\langle z_1,\ldots,z_m\rangle$ is given by $\rho(x_{ij})=\rho_{ij}$, and $\rho_{ij}$ is minus the $(i,j)$-entry of the inverse path matrix $P_{\beta}^{-1}$. Here $\FM$ and $\FG$ are as in \Cref{ssec:Cohn_loc_Gauss_dec}, and the map $\FM\lr\FG$ in the derived pushout is $\gamma:\FM\lr\FG$, as in \Cref{def:FM_and_FG}.$(iii)$.\\

In order to compute the derived pushout, it suffices to compute an underived pushout with a cofibration resolution of $\kk\langle z_1,\ldots,z_m\rangle$, as a dg algebra over $\FM$. Indeed, the pushout functor
\[
 \FG*_{\FM}(-):\FM\!\downarrow\!\DGA
 \longrightarrow \FG\!\downarrow\!\DGA
\]
is left adjoint to restriction along \(\gamma:\FM\to\FG\). Note that restriction preserves fibrations and trivial fibrations, and hence the pushout functor is left
Quillen. Therefore its total left-derived functor is computed by taking
any cofibrant replacement in \(\FM\!\downarrow\!\DGA\) and then forming
an (underived) pushout. That is, our task is to find a cofibrant replacement $Q_\beta\lr \kk\langle z_1,\ldots,z_m\rangle$ in \(\FM\!\downarrow\!\DGA\) and study \eqref{eq:derived_pushout} via

\begin{equation}\label{eq:derived_pushout2}
\FG*^{\mathbf L}_{\FM}\kk\langle z_1,\ldots,z_m\rangle\simeq\FG*_{\FM}Q_\beta
\end{equation}

For that, let us introduce the following:

\begin{definition}\label{def:cofibrant_replacement}
The dg algebra $Q_\beta$ is defined to be
\[
 Q_\beta:=\FM*\kk\langle z_1,\ldots,z_m,b_{11},\ldots,b_{ij},\ldots,b_{nn}\rangle,\quad i,j\in[1,n],
\]
where the generators $z_s$ and $b_{ij}$ have degrees $|z_s|=0$ and $|b_{ij}|=1$ and the differential is 
\[
 \partial z_s:=0,
 \qquad
 \partial b_{ij}:=x_{ij}-\rho_{ij}(z_1,\ldots,z_m).
\]
where $\rho_{ij}$ is minus the $(i,j)$-entry of the inverse path matrix $P_{\beta}^{-1}$.\hfill$\Box$
\end{definition}

\noindent We always consider the dg algebra $Q_\beta$ in \Cref{def:cofibrant_replacement} as an object of \(\FM\!\downarrow\!\DGA\), via the canonical map $\FM\lr Q_\beta$ to the first factor times the identity, i.e.~the canonical inclusion of the free-product factor $\FM$. Our next result is to show that $Q_\beta$ is the required cofibrant replacement:

\begin{proposition}
\label{lem:relative-resolution} The map of \(\FM\)-algebras
\[
 \varepsilon:Q_\beta\longrightarrow \kk\langle z_1,\ldots,z_m\rangle,
 \qquad
 z_s\longmapsto z_s,
 \quad x_{ij}\longmapsto \rho_{ij},
 \quad b_{ij}\longmapsto0,
\]
is a cofibrant replacement in
\(\FM\!\downarrow\!\DGA\).
\end{proposition}

\begin{proof}
First, the extension \(\FM\to Q_\beta\) is relative semi-free.  Indeed,
we start with \(\FM\), first adjoin the degree-zero generators
\(z_1,\ldots,z_m\), and then adjoin the degree-one generators \(b_{ij}\),
whose differentials are expressed in terms of the generators of $\FM$ and \(z_1,\ldots,z_m\). Therefore
\(\FM\to Q_\beta\) is a cofibration, and \(Q_\beta\) is cofibrant in the
under category \(\FM\!\downarrow\!\DGA\).\\

\noindent Second, we must show that \(\varepsilon\) is a quasi-isomorphism. For that, let us perform the following triangular change of degree-zero variables
\[
 v_{ij}:=x_{ij}-\rho_{ij}(z_1,\ldots,z_m).
\]
Since we can write
\[
 x_{ij}=v_{ij}+\rho_{ij}(z_1,\ldots,z_m),
\]
this is indeed an invertible change of free generators. Let us write $R:=\kk\langle z_1,\ldots,z_m\rangle$ to ease notation. In these new variables $v_{ij}$ we can write
\[
 Q_\beta\cong R\langle v_{ij},b_{ij}\rangle,\quad i,j\in[1,n]
 \qquad
 \partial b_{ij}=v_{ij},
 \quad
 \partial v_{ij}=0.
\]
\noindent Consider \(W\) to be the two-term chain complex with basis \(v_{ij}\) in degree zero, basis
\(b_{ij}\) in degree one, with differentials
\[
 \partial b_{ij}=v_{ij},
 \qquad
 \partial v_{ij}=0.
\]
Then, as dg algebras over the free algebra \(R\), we have an isomorphism
\[
 Q_\beta\cong T_R(R\otimes_{\kk}W\otimes_{\kk}R).
\]
Consequently, as a chain complex, we have the equality
\[
 Q_\beta=R\oplus\bigoplus_{q\geq1}K_q,
\]
where we have denoted $K_q:=
 R\otimes W\otimes R\otimes W\otimes\cdots\otimes W\otimes R$ containing exactly \(q\) copies of \(W\). Let us argue that each $K_q$ is contractible for $q\geq1$, from which we will conclude $\varepsilon$ is a quasi-isomorphism. For that, consider the chain homotopy \(h:W\to W[1]\) given by
\[
 h(v_{ij})=b_{ij},
 \qquad
 h(b_{ij})=0,
\]
so that $\partial h+h\partial=\id_W$. Similarly, for \(q\geq1\), consider the chain homotopy \(H_q:K_q\to K_q[1]\) obtained by applying \(h\) to the
first \(W\)-factor, so that
\[
 H_q(r_0\otimes w_1\otimes r_1\otimes\cdots\otimes w_q\otimes r_q)
 =r_0\otimes h(w_1)\otimes r_1\otimes\cdots\otimes w_q\otimes r_q.
\]
We claim that this is a chain null-homotopy for $K_q$:
\begin{equation}\label{eq:null_homotopy}
 \partial H_q+H_q\partial=\id_{K_q}.   
\end{equation}
Indeed, first note that the elements of \(R\) have degree zero. Now, in the tensor-product
differential, the terms in which the differential hits a later
\(W\)-factor cancel between \(\partial H_q\) and \(H_q\partial\), because
\(|h(w_1)|=|w_1|+1\). The remaining terms are
\((\partial h+h\partial)(w_1)\) and thus \eqref{eq:null_homotopy} follows. Thus, since the identity is homotopic to the zero map, every \(K_q\) is contractible for \(q\geq1\). Finally, the summand with
\(q=0\) is \(R\), and \(\varepsilon\) restricts to the identity on this
summand.  Therefore \(\varepsilon\) is a quasi-isomorphism.
\end{proof}

\subsubsection{Proof of \Cref{prop:derived-model}}
First, note that considering the leading principal submatrices of $P_\beta^{-1}$ or their negatives, by adding a minus sign to each of them, results in quasi-isomorphic derived localizations. Let us then apply \Cref{lem:relative-resolution} to obtain
\begin{align*}
 L_{\Sigma_\beta}(\kk\langle z_1,\ldots,z_m\rangle)
 &\simeq \FG*^{\mathbf L}_{\FM}\kk\langle z_1,\ldots,z_m\rangle\\
 &\simeq \FG*_{\FM}Q_\beta.
\end{align*}
As explained above, the last expression is an ordinary (underived) pushout, since \Cref{lem:relative-resolution} implies that \(Q_\beta\to \kk\langle z_1,\ldots,z_m\rangle\) is a
cofibrant replacement in \(\FM\!\downarrow\!\DGA\). Let us now compute this pushout and show it coincides with $\cG(\beta)$.  In \(\FG*_{\FM}Q_\beta\), the image of the generator
\(x_{ij}\) from \(Q_\beta\) is identified with its image from \(\FG\) under the map $\gamma$ in \Cref{def:FM_and_FG}.$(iii)$,
namely the entry \((DLU)_{ij}\).  Therefore the differential of \(b_{ij}\) becomes
\[
 \partial b_{ij}=(DLU)_{ij}-\rho_{ij}.
\]
%After eliminating the redundant symbols \(x_{ij}\), 
Thus the pushout has
precisely the generators and differential of \(\cG(\beta)\). Therefore there
exists the following isomorphism of dg algebras under \(\kk\langle z_1,\ldots,z_m\rangle\):
\[
 \FG*_{\FM}Q_\beta\cong\cG(\beta).
\]
Therefore we can complete the chain of equivalences
\begin{align*}
 L_{\Sigma_\beta}(\kk\langle z_1,\ldots,z_m\rangle)
 &\simeq \FG*^{\mathbf L}_{\FM}\kk\langle z_1,\ldots,z_m\rangle\\
 &\simeq \FG*_{\FM}Q_\beta\\
  &\simeq \cG(\beta).
\end{align*}
This implies the quasi-isomorphism \eqref{eq:LegendrianDGA_derived_Cohn}, as required.\hfill$\Box$

%%%%%%%%%%%%%%%%%%%%%%%%%%%%%%%%%%%%%%%%%%%
%%%%%%%%%%%%%%%%%%%%%%%%%%%%%%%%%%%%%%%%%%%
\subsection{ Derived Cohn localization over free algebras.}\label{ssec:derived_Cohn_underived} An important technical result we establish, which is strictly algebraic, will be that derived Cohn localizations over a free algebra are quasi-isomorphic to (underived) Cohn localizations. For underived Cohn localization, see \cite[Chapter 7]{CohnBook}, and for derived Cohn localization, see \Cref{ssec:Legendrian_DGA_derived_Cohn} above or \cite[Section 4.3]{BCL}.\\

\noindent Given a ring $R$ and $\Sigma$ a family of morphisms between finitely generated free right
\(R\)-modules, we denote by $R_\Sigma$ the (underived) universal Cohn localization along $\Sigma$, i.e.~inverting the maps in
\(\Sigma\). The natural map from $R$ to its Cohn localization $R_\Sigma$ will be denoted
\begin{equation}\label{eq:underived_Cohn}
\lambda_\Sigma:R\longrightarrow R_\Sigma
\end{equation}

\noindent Similarly, we denote by $L^{\mathrm{dga}}_\Sigma(R)$ the derived Cohn localization of $R$ along $\Sigma$, cf.~\cite[Definition 4.35]{BCL}. From a derived viewpoint, we regard both $R$ and $R_\Sigma$ as dg algebras concentrated in
homological degree \(0\), and $\lambda_\Sigma$ in \eqref{eq:underived_Cohn} as a morphism between them. As with any derived object, there exists a canonical morphism
\begin{equation}\label{eq:truncation_localizations}
t_0:L^{\mathrm{dga}}_\Sigma(R)\longrightarrow R_\Sigma
\end{equation}
from the derived Cohn localization to the underived one. The precise statement we prove reads as follows:

\begin{theorem}[Derived Cohn localization is underived over free algebras]\label{prop:ordinary-equals-derived-localization}
Let $R=\kk\langle z_1,\ldots,z_m\rangle$ be a free associative \(k\)-algebra and $\Sigma$ be a collection of morphisms between finitely generated free right \(R\)-modules. Then
\[
t_0:L^{\mathrm{dga}}_\Sigma(R)\longrightarrow R_\Sigma,
\]
is a quasi-isomorphism.
\end{theorem}

\noindent A direct consequence of \Cref{prop:ordinary-equals-derived-localization} is the vanishing of higher homology for derived Cohn localizations of free algebras:

\begin{corollary}
\label{cor:ordinary-equals-derived-localization2}
Let $R=\kk\langle z_1,\ldots,z_m\rangle$ be a free associative \(k\)-algebra and $\Sigma$ be a collection of morphisms between finitely generated free right \(R\)-modules. Then
\[
H_0\!\left(L^{\mathrm{dga}}_\Sigma(R)\right)\cong R_\Sigma,\qquad H_i\!\left(L^{\mathrm{dga}}_\Sigma(R)\right)=0,\quad \forall i\neq 0.
\]
\hfill$\Box$
\end{corollary}

\noindent The proof of \Cref{prop:ordinary-equals-derived-localization} occupies the rest of this subsection.

\begin{proof}[Proof of \Cref{prop:ordinary-equals-derived-localization}] The proof is structured in the following steps:

\begin{enumerate}[label=$(\roman*)$]
    \item First, we show that the map
\begin{equation}\label{eq:homological_epi}
R_\Sigma\otimes_R^{\mathbf L}R_\Sigma\longrightarrow R_\Sigma
\end{equation}
induced by multiplication is a quasi-isomorphism, where $R_\Sigma$ is seen as an $R$-module via the localization morphism $\lambda_\Sigma$ in \eqref{eq:underived_Cohn}.

\item Second, we use $(i)$ to prove an isomorphism
\begin{equation}\label{eq:derived_localization}
L^{\mathrm{Mod(R)}}_\Sigma(R)\simeq R\otimes_R^{\mathbf L}R_\Sigma
\end{equation}
of $R$-modules, where $L^{\mathrm{Mod(R)}}_\Sigma(R)$ is the underlying $R$-module of the dg algebra $L^{\mathrm{dga}}_\Sigma(R)$.

\item Conclude from $(ii)$ the required equivalence $L^{\mathrm{dga}}_\Sigma(R)\simeq R_\Sigma$ induced by $t_0$.
\end{enumerate}

\noindent Let us provide the details and justification for each such step. First, for Step $(i)$, let us use that \cite[Corollary 4.3]{CohnBook} implies that the free associative algebra
\(R\) is right and left hereditary. Since \(\lambda_\Sigma:R\to R_\Sigma\) is a
universal localization of a hereditary ring, \cite[Theorem 6.1]{KrauseStovicek} implies that $\lambda$ is a homological ring
epimorphism, i.e.~the multiplication map \eqref{eq:homological_epi} is a quasi-isomorphism.\\

For Step $(ii)$, the goal is to show that the derived extension of scalars functor
\[
-\otimes_R^{\mathbf L}R_\Sigma:
\mathsf D(\operatorname{Mod}(R_\Sigma))
\longrightarrow
\mathsf D(\operatorname{Mod}(R))
\]
is actually the correct derived localization at the level of $R$-modules, from which \eqref{eq:derived_localization} will follow. To ease notation, let us denote such functor by $F:=-\otimes_R^{\mathbf L}R_\Sigma$, and recall that we denote by $\lambda_\Sigma:R\to R_\Sigma$ the universal localization. By Step~(i), $\lambda$ is a homological epimorphism, and so the induced restriction
\[
 (\lambda_\Sigma)_*:D(R_\Sigma)\longrightarrow D(R)
\]
is fully faithful and $F=-\otimes_R^{\mathbf L}R_\Sigma$ is left adjoint to $(\lambda_\Sigma)_*$.  The essential image consists of complexes whose homology modules lie in the essential image of $\Mod(R_\Sigma)\to\Mod(R)$.  By the universal property of universal localization, cf.~\cite[Section 6]{KrauseStovicek}, these are exactly the modules $M$ for which $\operatorname{Hom}_R(\sigma,M)$ is an isomorphism for every $\sigma\in\Sigma$.  Since the source and target of each $\sigma$ are finitely generated projective modules, a complex $X$ is derived $\Sigma$-local iff every $H_i(X)$ has this property.  Hence the essential image of $(\lambda_\Sigma)_*$ is precisely the full subcategory of derived $\Sigma$-local objects.  Therefore the adjunction unit $X\to(\lambda_\Sigma)_*F(X)$ exhibits $F$ as a Bousfield localization, as required. In particular, we conclude that
\[
 L^{\Mod(R)}_\Sigma(R)\simeq R\otimes_R^{\mathbf L}R_\Sigma\simeq R_\Sigma.
\]

For Step $(iii)$, the final step, consider the canonical truncation morphism
$$t_0:L^{\mathrm{dga}}_\Sigma(R)\longrightarrow R_\Sigma$$
from \eqref{eq:truncation_localizations}. By Step $(ii)$, the underlying $R$-module morphism is an isomorphism in the derived category $D(R)$, and thus a quasi-isomorphism of $R$-modules. Then \cite[Theorem 4.38]{BCL} implies that $t_0$ itself must be a quasi-isomorphism of dg algebras, since we have proven that the underlying $R$-module morphism is a quasi-isomorphism.
\end{proof}
\color{black}

%%%%%%%%%%%%%%%%%%%%%%%%%%%%%%%%%%%%%%%%%%%
%%%%%%%%%%%%%%%%%%%%%%%%%%%%%%%%%%%%%%%%%%%
\subsection{ Proof of Main Theorem.}\label{ssec:proof_of_main_theorem} It follows from \Cref{prop:dga-change}, \Cref{prop:derived-model} and \Cref{prop:ordinary-equals-derived-localization} that we have the chain of quasi-isomorphisms in \eqref{eq:sequence_qiso}:

\begin{equation*}
\SA(\La_\beta,\ST_{s})\simeq\cG(\beta)\simeq L^{\mathrm{dga}}_{\Sigma_\beta}(\kk\langle z_1,\ldots,z_m\rangle)\simeq \kk\langle z_1,\ldots,z_m\rangle_{\Sigma_\beta}.
\end{equation*}

\noindent Since the higher homology of $\kk\langle z_1,\ldots,z_m\rangle_{\Sigma_\beta}$ vanishes, as it is concentrated in degree 0, the higher homology of $\SA(\La_\beta,\ST_s)$ also vanishes. This establishes the required vanishing result for $\SA(\La_\beta,\ST_{s})$ over a field $\kk$. In order to deduce the vanishing of the higher homology of $\SA(\La_\beta,\ST_{s})$ over $\ZZ$, it suffices to use the following two facts:

\begin{enumerate}
    \item[(i)] By construction, the chain groups of the chain complex underlying the $\ZZ$-graded Legendrian dg algebra are free abelian groups, and the Legendrian dg algebra behaves well under base change.

    \item[(ii)] In general, for a non-negatively graded chain complex $C$ of free abelian groups with vanishing higher homologies $H_{*\geq1}(C\otimes \mathbb{Q})=0$ and $H_{*\geq1}(C\otimes \mathbb{F}_p)=0$, for all primes $p\in\mathbb{N}$, we also must have $H_{*\geq1}(C)=0$. Indeed, since $\mathbb{Q}$ is flat, the hypotheses imply that each $H_\ell(C)$ must be a torsion abelian group, $\ell\geq1$. By the universal coefficient theorem, the vanishing $H_{\ell+1}(C\otimes \mathbb{F}_p)=0$ implies $\mbox{Tor}_1^\ZZ(H_\ell(C),\mathbb{F}_p)=0$ for all primes $p\in\mathbb{N}$, and thus $H_\ell(C)=0$.
\end{enumerate}
By $(i)$ and the proven vanishing over an arbitrary field $\kk$, we can apply $(ii)$. This implies the required vanishing result for $\SA(\La_\beta,\ST_{s})$ over $\ZZ$.\\

\noindent Finally, $\SA(\La_\beta,\ST_{s})$ is endowed with one basepoint per strand, whereas $\SA(\La_\beta)$ in the statement of the Main Theorem is endowed with one basepoint per component of $\La_\beta$. Now, over a field $\kk$, adding a basepoint to a component of $\Lambda_\beta$ that already has a basepoint only changes the dg algebra by a free unital product with $(\kk\langle s^{\pm1}\rangle,\dd s=0)$, in the dg-category of dg algebras, see e.g.~\cite[Remark 3.1]{CasalsNg}. Therefore, over $\kk$, the vanishing of the higher homology of $\SA(\La_\beta,\ST_{s})$ implies the vanishing of the higher homology of $\SA(\La_\beta)$. The argument above lifts such vanishing over $\ZZ$, concluding the required vanishing of the higher Legendrian homology of $\SA(\La_\beta)$ over $\ZZ$.

\hfill$\Box$

%%%%%%%%%%%%%%%%%%%%%%%%%%%%%%%%%%%%%%%%%%%
%%%%%%%%%%%%%%%%%%%%%%%%%%%%%%%%%%%%%%%%%%%
%%%%%%%%%%%%%%%%%%%%%%%%%%%%%%%%%%%%%%%%%%%
\section{Sample computation of Hochschild homology for max-tb $(2,m)$-torus links}\label{sec:examples_applications}

There are a number of advantages to having obtained an underived model for the Legendrian contact dg algebra of $\La_\beta$, as in the Main Theorem. For instance, it becomes significantly simpler to compute derived invariants, as we can use classical homological methods. To illustrate this, we will now fully compute the Hochschild homology of the Legendrian contact dg algebra of $\La_\beta$ for any 2-stranded braid $\beta$ and over any field $\kk$.\\

In order to be as useful as possible, we will perform all computations with one basepoint per component. This distinguishes the computations and results for $\beta=\sigma_1^m$ depending on whether $m$ is odd, in which case $\La_\beta$ is a knot, or $m$ is even, in which case $\La_\beta$ is a 2-component link. For the former, we will keep $t:=t_1$ as a basepoint and specialize to $t_2=1$, whereas for the latter, the 2-component link case, we keep both $t_1$ and $t_2$ as basepoints. The main result we establish can be summarized as:

\begin{proposition}\label{prop:Hochschild_2stranded} Let $\beta=\sigma_1^m\in\mbox{Br}_2^+$ be a 2-stranded positive braid word and $\SA(\La_\beta)$ the Legendrian dg algebra associated to $\La_\beta$, with one basepoint per component and over a field $\kk$. Then the higher Hochschild homology vanishes:
$$\HH_q(\SA(\La_\beta))=0,\qquad q\geq2$$
and the remaining $\HH_0(\SA(\La_\beta))$ and $\HH_1(\SA(\La_\beta))$ can be computed explicitly.
\end{proposition}

\noindent The explicit computation of $HH_0(\SA(\La_\beta))$ and $HH_1(\SA(\La_\beta))$ can be found in \Cref{prop:HH_even_case} and \Cref{prop:HH_odd_case} below. Similar results can be obtained for variants of Hochschild homology. For instance, it follows from \Cref{prop:Hochschild_2stranded} that the relative negative cyclic homology associated to the extension of scalars functor $F:\Perf(C_{-*}(\Omega\La_\beta))\lr\Perf(\SA(\La_\beta))$ vanishes in degree 3 and higher for non-trivial 2-stranded braids. For degree 2, one can then use \cite[Section 3.1]{CGGS24} to explicitly describe the relative Calabi-Yau structure carried by the extension of scalars functor. Let us now focus on computing $HH_*(\SA(\La_\beta))$ for 2-stranded braids and prove \Cref{prop:Hochschild_2stranded}.

%%%%%%%%%%%%%%%%%%%%%%%%%%%%%%%%%%%%%%%%%%%
%%%%%%%%%%%%%%%%%%%%%%%%%%%%%%%%%%%%%%%%%%%
\subsection{ Basic facts about Hochschild homology}\label{ssec:hochschild_review} We include a brief reminder of Hochschild homology, as follows. Let $A$ be a unital $\kk$-algebra and define $A^e:=A\otimes_{\kk}A^{\op}$, so that an $A$-bimodule is equivalent to a left $A^e$-module.  Since $A$ is projective over $\kk$, the Hochschild homology groups of $A$ can be computed as:
\begin{equation}\label{eq:HHtor}
\HH_q(A)=\Tor_q^{A^e}(A,A).
\end{equation}
%This is proved by tensoring the bar resolution with $A$; see Loday
%\cite[Proposition~1.1.12 and Proposition~1.1.13]{Loday}.  The normalized Hochschild complex is obtained by quotienting the bar complex by the acyclic degenerate subcomplex; see \cite[Section~1.1.14, especially Proposition~1.1.15]{Loday}.
Therefore, if $P_\bullet\longrightarrow A$ is any projective $A^e$-resolution, then the homology of $A\otimes_{A^e}P_\bullet$ computes $\HH_\bullet(A)$. We will use the notation
\[
 A_\natq:=A/[A,A],\qquad
 \Om^1_{A,\natq}:=\Om^1_{A}/[A,\Om^1_{A}],
\]
where $\Om^1_{A}$ is the module of non-commutative K\"ahler differentials of $A$.

%%%%%%%%%%%%%%%%%%%%%%%%%%%%%%%%%%%%%%%%%%%
%%%%%%%%%%%%%%%%%%%%%%%%%%%%%%%%%%%%%%%%%%%
\subsection{ Legendrian dg algebras for 2-stranded braids}\label{ssec:dga_2stranded_review} For the case of a 2-stranded braid word $\beta=\sigma_1^{m}\in\mbox{Br}^+_2$, $m\in\mathbb{N}$, the $0$th homology $A_\beta=H_0(\SA(\La_\beta);\kk)$ of $\SA(\La_\beta)$ has a particularly simple form, see e.g.~\cite[Section 3]{chantraine2019representations}. Specifically, consider the left and right continuants, respectively denoted $K_{i}$ and $\widetilde K_i$, given recursively by
\begin{align*}
 K_{-1}&:=0,&K_0&:=1,&K_j&:=z_jK_{j-1}+K_{j-2},\\
 \widetilde K_{-1}&:=0,&\widetilde K_0&:=1,&
 \widetilde K_j&:=\widetilde K_{j-1}z_j+\widetilde K_{j-2}.
\end{align*}
Thus $K_1=\widetilde K_1=z_1$, and the continuants are understood as elements in the free algebra $\kk\langle z_1,\ldots,z_m\rangle$. It follows from \Cref{cor:main1}, combined with \cite[Section 3]{chantraine2019representations} or \cite[Section 5.1]{CasalsNg}, that the Legendrian contact dg algebra $\SA(\La_\beta)$ of $\beta=\sigma_1^{m}\in\mbox{Br}^+_2$ admits the following model:

\begin{lemma}\label{lemma:Legendrian_dga_2_strands}
Let $\beta=\sigma_1^{m}\in\mbox{Br}^+_2$ and consider the Legendrian dg algebra $\SA(\La_\beta)$ over a field $\kk$, with one basepoint per component. Then $\SA(\La_\beta)$ is quasi-isomorphic to
\begin{equation}\label{eq:Legendrian_dga_2_strands}
\SA(\La_\beta)\simeq\begin{cases}
\kk\langle z_1,\ldots,z_m\rangle[K_m^{-1}] & \mbox{ if $m$ is even},\\
\kk\langle z_1,\ldots,z_m\rangle/(1+\widetilde K_m), & \mbox{ if $m$ is odd}.
\end{cases}
\end{equation}
In particular, the Hochschild homology $HH_\bullet(\SA(\La_\beta))$ is isomorphic to the Hochschild homology of the corresponding $($underived$)$ algebra on the right-hand side of \eqref{eq:Legendrian_dga_2_strands}.\hfill$\Box$
\end{lemma}

\noindent In the statement of \Cref{lemma:Legendrian_dga_2_strands} we have used $\kk\langle z_1,\ldots,z_m\rangle[K_m^{-1}]$ to denote the universal localization, adjoining a two-sided inverse to $K_m$. Note that for $m$ even, the basepoints $t_1$ and $t_2$ in $\SA(\La_\beta)$ map to
$$t_1\mapsto-K_m,\qquad t_2\mapsto-\widetilde K_m^{-1},$$
under the identification in \eqref{eq:Legendrian_dga_2_strands}. It is thus relevant to emphasize that $\widetilde K_m$ is invertible because $-K_m$ is invertible, as follows from their recursive definition. Now, due to the parity dependence, which results from the fact that $\La_\beta$ has one or two components depending on whether $m$ is odd or even, we will now separate the computation of Hochschild homology for $m$ even and odd.

%%%%%%%%%%%%%%%%%%%%%%%%%%%%%%%%%%%%%%%%%%%
%%%%%%%%%%%%%%%%%%%%%%%%%%%%%%%%%%%%%%%%%%%
\subsection{ Computation for $m$ even}\label{ssec:HH_m_even} Following \Cref{lemma:Legendrian_dga_2_strands}, we ease notation and set $A:=\kk\langle z_1,\ldots,z_m\rangle[K_m^{-1}]$, where $K_m$ is the continuant as defined above. To describe its Hochschild homology, it is convenient to consider the map
\begin{equation}\label{eq:b_map}
 b_m:A^m\longrightarrow A,\qquad
 b_m(a_1,\ldots,a_m)=\sum_{i=1}^m[a_i,z_i].
\end{equation}
We now prove the following:

\begin{proposition}\label{prop:HH_even_case} Let $A:=\kk\langle z_1,\ldots,z_m\rangle[K_m^{-1}]$, then
\[
 \HH_q(A)\cong
 \begin{cases}
 A_\natq,&q=0,\\
 \ker(b_m:A^m\to A),&q=1,\\
 0,&q\ge2.
 \end{cases}
\]
\end{proposition}

\begin{proof}
First, we claim that there is an isomorphism of $A$-bimodules
\begin{equation}\label{eq:omega-free}
\Om^1_{A}\cong\bigoplus_{i=1}^m A\,dz_i\,A,
\end{equation}

and thus $\Om^1_A$ is free over $A^e$. Indeed, for any $A$-bimodule $M$, any arbitrary choices $D(z_i)\in M$ determine a
unique derivation on $\kk\langle z_1,\ldots,z_m\rangle$. It extends uniquely across the universal
localization because
\[
 D(K_m^{-1})=-K_m^{-1}D(K_m)K_m^{-1}.
\]
Thus we have an isomorphism $\Der_\kk(A,M)\cong M^m$, naturally in $M$.  The left-hand side $\Der_\kk(A,M)$ of such isomorphism is represented
by $\Om^1_A$, see e.g.~\cite[Section~2, Proposition~2.4]{CQ}, and the right side $M^m$ by the bimodule in \eqref{eq:omega-free}. Therefore \eqref{eq:omega-free} follows from the Yoneda Lemma.\\

Now, in general, for an algebra $A$, the universal differential sequence
\begin{equation}\label{eq:univseq}
0\longrightarrow\Om^1_A\xrightarrow{j}A\otimes_\kk A
\xrightarrow{\mu}A\longrightarrow0,
\qquad j(a\,db)=ab\otimes1-a\otimes b,
\end{equation}

is exact, where $\mu$ is the algebra product, cf.~e.g.~\cite[Proposition~2.5]{CQ}.  By 
\eqref{eq:omega-free} above, $\Om^1_A$ is free, and thus projective, over $A^e$. Hence \eqref{eq:univseq} is a projective $A^e$-resolution
of $A$ of length one. (See also \cite[Proposition~3.3]{CQ}.) In order to compute Hochschild homology as in \eqref{eq:HHtor}, we tensor the resolution part of \eqref{eq:univseq} with $A$ over
$A^e$ to obtain
\begin{equation}\label{eq:evencomplex}
0\longrightarrow A^m\xrightarrow{b_m}A\longrightarrow0.
\end{equation}

Indeed, we obtain \eqref{eq:evencomplex} because $A\otimes_{A^e}\Om^1_A\simeq\Om^1_{A,\natq}$ and cyclic tensoring sends
$j(a\,db)$ to $ab-ba$.  Thus the differential is indeed $b_m$ and so taking
homology in \eqref{eq:evencomplex} proves the statement of the proposition.
\end{proof}
\color{black}

%%%%%%%%%%%%%%%%%%%%%%%%%%%%%%%%%%%%%%%%%%%
%%%%%%%%%%%%%%%%%%%%%%%%%%%%%%%%%%%%%%%%%%%
\subsection{ Computation for $m$ odd}\label{ssec:HH_m_odd} Similarly to \Cref{ssec:HH_m_even}, we use \Cref{lemma:Legendrian_dga_2_strands} and ease notation by setting $A:=\kk\langle z_1,\ldots,z_m\rangle/(1+\widetilde K_m)$, where $\widetilde K_m$ is the $m$th continuant as defined above. To describe its Hochschild homology, first write every monomial occurrence of $z_i$ in $1+\widetilde K_m$ as $p z_iq$, with its coefficient implicitly understood, and define the map
\begin{equation}\label{eq:Jm_map}
J_m:A\longrightarrow A^m,\qquad J_m(a)_i:=\sum_{p z_iq\in 1+\widetilde K_m}qap,
\end{equation}
where the sum runs over all the monomials of $1+\widetilde K_m$. The map $b_m$ defined by \eqref{eq:b_map}, now with $A:=\kk\langle z_1,\ldots,z_m\rangle/(1+\widetilde K_m)$, then descends to a map
$$\bar b_m:A^m/\im J_m\to A.$$
In this case, we will now prove the following:

\begin{proposition}\label{prop:HH_odd_case} Let $A:=\kk\langle z_1,\ldots,z_m\rangle/(1+\widetilde K_m)$, then
\[
 \HH_q(A)\cong
 \begin{cases}
 A_\natq,&q=0,\\
 \ker\!\left(\bar b_m:A^m/\im J_m\to A\right),&q=1,\\
 0,&q\ge2,
 \end{cases}
\]
\end{proposition}

In contrast to \Cref{prop:HH_even_case}, the proof of \Cref{prop:HH_odd_case} requires a few more steps, as there is no projective resolution of $A$ over $A^e$ as simple as \eqref{eq:univseq} in the odd $m$ case. Before we prove \Cref{prop:HH_odd_case}, we will construct a 3-step free resolution, as follows:

\begin{lemma}\label{prop:oddresolution} Let $A:=\kk\langle z_1,\ldots,z_m\rangle/(1+\widetilde K_m)$, and consider the $\kk$-vector space $V$ spanned by the $m$ vectors $z_1,\ldots,z_m$. Then
\begin{equation}\label{eq:oddresolution}
0\to A\otimes A\xrightarrow{\partial_2}
A\otimes_\kk V \otimes_\kk A
\xrightarrow{\partial_1}A\otimes A\xrightarrow{\mu}A\to0,
\end{equation}

defined by $\partial_1(1\otimes z_i\otimes1):=z_i\otimes1-1\otimes z_i$ and
$$
 \partial_2(1\otimes1):=
 \sum_{i=1}^m\sum_{p z_iq\in\widetilde K_m}p\otimes z_i\otimes q,
$$
is a free $A^e$-resolution of $A$. Here the internal sum for $\partial_2$ runs over all monomials of $\widetilde K_m$.
\end{lemma}
\begin{proof}
First, note that the composition $\partial_1\circ\partial_2$ vanishes, as it coincides with 
$$(1+\widetilde K_m)\otimes1-1\otimes (1+\widetilde K_m)=0.$$
To show that \eqref{eq:oddresolution} is indeed a resolution of $A$ by $A^e$-modules, we claim that the set of words not containing $z_1\cdots z_m$ forms a $\kk$-basis of $A$. Indeed, the polynomial $\widetilde K_m$ has the unique highest-degree monomial $z_1\cdots z_m$, as every other monomial has degree at
most $m-2$. (We order words first by degree and then lexicographically.) Since the letters in $z_1\cdots z_m$
are pairwise distinct, $z_1\cdots z_m$ has no proper self-overlap and the Diamond
Lemma \cite[Theorem~1.2]{Bergman} applies to show that the set of words not containing $z_1\cdots z_m$ form a $\kk$-basis of $A$. In fact, such a $\kk$-basis is the irreducible-word basis, with the sole $1$-ambiguity being $z_1\cdots z_m$, and without any higher ambiguities, as there are no overlaps. Now the required exactness of the $A^e$-module complex above follows from \cite[Theorem~4.1]{CS}. Note that we can apply their result, as their lower-order hypothesis holds because $1+\widetilde K_m-z_1\cdots z_m$ has degree at most $m-2$ if $m\geq3$. The case of $m=1$ is verified directly. In conclusion, we obtain that \eqref{eq:oddresolution} is a resolution of $A$ by free $A^e$-modules.
\end{proof}

\begin{proof}[Proof of \Cref{prop:HH_odd_case}]
By \Cref{prop:oddresolution}, the exact sequence \eqref{eq:oddresolution} is a resolution of $A$ by $A^e$-modules. By tensoring this resolution in \eqref{eq:oddresolution} cyclically we obtain
\begin{equation}\label{eq:oddcomplex}
0\longrightarrow A\xrightarrow{J_m}A^m\xrightarrow{b_m}A\longrightarrow0,
\end{equation}
because $p\otimes z_i\otimes q$ sends $a$ to $qap$ in the $i$th coordinate. Here the $b_m$ and $J_m$ maps are respectively defined by \eqref{eq:b_map} and \eqref{eq:Jm_map}. Let us now argue that $J_m:A\to A^m$ is injective.\\

\noindent For that we can use the word-length filtration, as follows. As established in the proof of \Cref{prop:oddresolution}, the set of words not containing $z_1\cdots z_m$ is a $\kk$-basis of $A$, and thus we have an isomorphism
$$\gr A\cong \kk\langle z_1,\ldots,z_m\rangle/(z_1\cdots z_m).$$
\noindent Now, the top-degree part of the first coordinate of
$J_m(a)$ is $(z_2\cdots z_m)a$ and left multiplication by $z_2\cdots z_m$ is
injective on this word basis: prefixing a normal word cannot create
$z_1\cdots z_m$, since the prefix has no initial $z_1$. Thus the leading part of $J_m$ is injective. Since the filtration is exhaustive, bounded below, and separated, we conclude that the entire map $J_m$ is injective.\\

From the injectivity of $J_m:A\to A^m$ and \eqref{eq:oddcomplex} we conclude that $\HH_q(A)=0$ for any $q\ge2$. For the $q=1$ case, we can quotient \eqref{eq:oddcomplex} by the acyclic subcomplex
$[A\xrightarrow{\sim}\im J_m]$, which leaves the two-term complex
\begin{equation}\label{eq:oddtwoterm}
0\longrightarrow A^m/\im J_m\xrightarrow{\bar b_m}A\longrightarrow0.
\end{equation}

Taking homology of \eqref{eq:oddtwoterm} proves the statement of \Cref{prop:HH_odd_case}.
\end{proof}

\bibliographystyle{plain}
\bibliography{main}

@article{Bergman,
  author  = {Bergman, G. M.},
  title   = {The diamond lemma for ring theory},
  journal = {Adv. Math.},
  volume  = {29},
  number  = {2},
  year    = {1978},
  pages   = {178--218}
}

@article{CS,
  author  = {Chouhy, S. and Solotar, A.},
  title   = {Projective resolutions of associative algebras and ambiguities},
  journal = {J. Algebra},
  volume  = {432},
  year    = {2015},
  pages   = {22--61}
}

@article{CQ,
  author  = {Cuntz, J. and Quillen, D.},
  title   = {Algebra extensions and nonsingularity},
  journal = {J. Amer. Math. Soc.},
  volume  = {8},
  number  = {2},
  year    = {1995},
  pages   = {251--289}
}

@article{chantraine2019representations,
  title={Representations, sheaves, and Legendrian (2,m) torus links},
  author={Chantraine, Baptiste and Ng, Lenhard and Sivek, Steven},
  journal={Journal of the London Mathematical Society},
  volume={100},
  number={1},
  pages={41--82},
  year={2019},
  doi={10.1112/jlms.12204},
  publisher={Wiley Online Library}
}

@article {Sylvan2016,
    AUTHOR = {Sylvan, Zachary},
     TITLE = {On partially wrapped {F}ukaya categories},
   JOURNAL = {J. Topol.},
  FJOURNAL = {Journal of Topology},
    VOLUME = {12},
      YEAR = {2019},
    NUMBER = {2},
     PAGES = {372--441},
      ISSN = {1753-8416,1753-8424},
   MRCLASS = {53D37},
  MRNUMBER = {3911570},
MRREVIEWER = {Xin\ Jin},
       DOI = {10.1112/topo.12088},
       URL = {https://doi.org/10.1112/topo.12088},
}

@article {Ganatraetal2023,
    AUTHOR = {Ganatra, Sheel and Pardon, John and Shende, Vivek},
     TITLE = {Sectorial descent for wrapped {F}ukaya categories},
   JOURNAL = {J. Amer. Math. Soc.},
  FJOURNAL = {Journal of the American Mathematical Society},
    VOLUME = {37},
      YEAR = {2024},
    NUMBER = {2},
     PAGES = {499--635},
      ISSN = {0894-0347,1088-6834},
   MRCLASS = {53D37 (53D40 57R17)},
  MRNUMBER = {4695507},
MRREVIEWER = {Benjamin\ Gammage},
       DOI = {10.1090/jams/1035},
       URL = {https://doi.org/10.1090/jams/1035},
}

@article{FuchsIshkhanov2004,
  author = {Fuchs, Dmitry B. and Ishkhanov, T.},
  title = {Invariants of {Legendrian} knots and decompositions of front diagrams},
  journal = {Moscow Mathematical Journal},
  volume = {4},
  number = {3},
  pages = {707--717},
  year = {2004}
}

@article{Sabloff2005,
  author = {Sabloff, Joshua M.},
  title = {Augmentations and rulings of {Legendrian} knots},
  journal = {International Mathematics Research Notices},
  volume = {2005},
  number = {19},
  pages = {1157--1180},
  year = {2005}
}

@article{MelvinShrestha2005,
  author = {Melvin, Paul and Shrestha, Sumana},
  title = {The nonuniqueness of {Chekanov} polynomials of {Legendrian} knots},
  journal = {Geometry \& Topology},
  volume = {9},
  number = {3},
  pages = {1221--1252},
  year = {2005}
}

@article{NgSabloff2006,
  author = {Ng, Lenhard and Sabloff, Joshua M.},
  title = {The correspondence between augmentations and rulings for {Legendrian} knots},
  journal = {Pacific Journal of Mathematics},
  volume = {224},
  number = {1},
  pages = {141--150},
  year = {2006}
}

@article{FuchsRutherford2011,
  author = {Fuchs, Dmitry and Rutherford, Dan},
  title = {Generating families and {Legendrian} contact homology in the standard contact space},
  journal = {Journal of Topology},
  volume = {4},
  number = {1},
  pages = {190--226},
  year = {2011}
}

@article{NgRutherford2013,
  author = {Ng, Lenhard and Rutherford, Dan},
  title = {Satellites of {Legendrian} knots and representations of the {Chekanov}-{Eliashberg} algebra},
  journal = {Algebraic \& Geometric Topology},
  volume = {13},
  number = {5},
  pages = {3047--3073},
  year = {2013}
}

@article{HenryRutherford2015,
  author = {Henry, Michael B. and Rutherford, Dan},
  title = {Ruling polynomials and augmentations over finite fields},
  journal = {Journal of Topology},
  volume = {8},
  number = {1},
  pages = {1--37},
  year = {2015}
}

@article{Leverson2016,
  author = {Leverson, Caitlin},
  title = {Augmentations and rulings of {Legendrian} knots},
  journal = {Journal of Symplectic Geometry},
  volume = {14},
  number = {4},
  pages = {1089--1143},
  year = {2016}
}

@article{Leverson2017,
  author = {Leverson, Caitlin},
  title = {Augmentations and rulings of {Legendrian} links in ${\#}_k(S^1 \times S^2)$},
  journal = {Pacific Journal of Mathematics},
  volume = {288},
  number = {2},
  pages = {381--423},
  year = {2017}
}

@article{NgRutherfordShendeSivekZaslow2017,
  author = {Ng, Lenhard and Rutherford, Dan and Shende, Vivek and Sivek, Steven and Zaslow, Eric},
  title = {The cardinality of the augmentation category of a {Legendrian} link},
  journal = {Mathematical Research Letters},
  volume = {24},
  number = {6},
  pages = {1845--1874},
  year = {2017}
}

@article{Traynor2001,
  author = {Traynor, Lisa},
  title = {Generating function polynomials for {Legendrian} links},
  journal = {Geometry \& Topology},
  volume = {5},
  pages = {719--760},
  year = {2001}
}

@article {CasalsGao2022,
    AUTHOR = {Casals, Roger and Gao, Honghao},
     TITLE = {Infinitely many {L}agrangian fillings},
   JOURNAL = {Ann. of Math. (2)},
  FJOURNAL = {Annals of Mathematics. Second Series},
    VOLUME = {195},
      YEAR = {2022},
    NUMBER = {1},
     PAGES = {207--249},
      ISSN = {0003-486X,1939-8980},
   MRCLASS = {57K33 (53D10)},
  MRNUMBER = {4358415},
MRREVIEWER = {Youlin\ Li},
       DOI = {10.4007/annals.2022.195.1.3},
       URL = {https://doi.org/10.4007/annals.2022.195.1.3},
}

@article {CasalsGao2023,
    AUTHOR = {Casals, Roger and Gao, Honghao},
     TITLE = {A {L}agrangian filling for every cluster seed},
   JOURNAL = {Invent. Math.},
  FJOURNAL = {Inventiones Mathematicae},
    VOLUME = {237},
      YEAR = {2024},
    NUMBER = {2},
     PAGES = {809--868},
      ISSN = {0020-9910,1432-1297},
   MRCLASS = {53D12 (13F60 57K33)},
  MRNUMBER = {4768635},
       DOI = {10.1007/s00222-024-01268-y},
       URL = {https://doi.org/10.1007/s00222-024-01268-y},
}

@article {EliashbergPolterovich1996,
    AUTHOR = {Eliashberg, Y. and Polterovich, L.},
     TITLE = {Local {L}agrangian {$2$}-knots are trivial},
   JOURNAL = {Ann. of Math. (2)},
  FJOURNAL = {Annals of Mathematics. Second Series},
    VOLUME = {144},
      YEAR = {1996},
    NUMBER = {1},
     PAGES = {61--76},
      ISSN = {0003-486X,1939-8980},
   MRCLASS = {58F05 (57Q45 57R40)},
  MRNUMBER = {1405943},
MRREVIEWER = {Serge\ L.\ Tabachnikov},
       DOI = {10.2307/2118583},
       URL = {https://doi.org/10.2307/2118583},
}

@article{EtnyreHonda2005,
  author  = {Etnyre, John and Honda, Ko},
  title   = {Cabling and transverse simplicity},
  journal = {Annals of Mathematics},
  year    = {2005},
  volume  = {162},
  pages   = {1305--1333},
  doi     = {10.4007/annals.2005.162.1305}
}

@article{EliashbergFraser2009,
  author  = {Eliashberg, Yakov and Fraser, Maia},
  title   = {Topologically trivial Legendrian knots},
  journal = {Journal of Symplectic Geometry},
  year    = {2009},
  volume  = {7},
  pages   = {77--127},
  doi     = {10.4310/jsg.2009.v7.n2.a4}
}

@article{Chekanov2002,
  author  = {Chekanov, Yuri},
  title   = {Differential algebra of Legendrian links},
  journal = {Inventiones Mathematicae},
  year    = {2002},
  volume  = {150},
  pages   = {441--483},
  doi     = {10.1007/s002220200212}
}

@article{Eliashberg2000,
  author  = {Eliashberg, Y. and Givental, A. and Hofer, H.},
  title   = {Introduction to Symplectic Field Theory},
  journal = {Visions in Mathematics},
  year    = {2000},
  pages   = {560--673},
  doi     = {10.1007/978-3-0346-0425-3_4}
}

@article {Ng2008,
    AUTHOR = {Ng, Lenhard},
     TITLE = {Rational symplectic field theory for {L}egendrian knots},
   JOURNAL = {Invent. Math.},
  FJOURNAL = {Inventiones Mathematicae},
    VOLUME = {182},
      YEAR = {2010},
    NUMBER = {3},
     PAGES = {451--512},
      ISSN = {0020-9910,1432-1297},
   MRCLASS = {53D42 (55P50 57M25 57M27)},
  MRNUMBER = {2737704},
MRREVIEWER = {David\ E.\ Hurtubise},
       DOI = {10.1007/s00222-010-0265-8},
       URL = {https://doi.org/10.1007/s00222-010-0265-8},
}

@article{BCL,
  author  = {C. Braun and J. Chuang and A. Lazarev},
  title   = {Derived localisation of algebras and modules},
  journal = {Adv. Math.},
  volume  = {328},
  year    = {2018},
  pages   = {555--622},
}

@article{CGGS24,
  author  = {R. Casals and E. Gorsky and M. Gorsky and J. Simental},
  title   = {Algebraic weaves and braid varieties},
  journal = {Amer. J. Math.},
  volume  = {146},
  number  = {6},
  year    = {2024},
  pages   = {1469--1576}
}

@article{CasalsNg,
  author  = {R. Casals and L. Ng},
  title   = {Braid loops with infinite monodromy on the Legendrian contact {DGA}},
  journal = {J. Topol.},
  volume  = {15},
  number  = {4},
  year    = {2022},
  pages   = {1927--2016},
}

@book{CohnBook,
  author    = {P. M. Cohn},
  title     = {Free Rings and Their Relations},
  edition   = {2nd},
  series    = {London Mathematical Society Monographs},
  volume    = {19},
  publisher = {Academic Press},
  address   = {London},
  year      = {1985}
}

@article{etnyre_ng_2020,
  author  = {J. B. Etnyre and L. Ng},
  title   = {Legendrian contact homology in {$\mathbb{R}^3$}},
  journal = {Surveys in Differential Geometry},
  volume  = {25},
  year    = {2020},
  pages   = {103--161}
}

@article{KrauseStovicek,
  author  = {H. Krause and J. \v{S}\v{t}ov\'{i}\v{c}ek},
  title   = {The telescope conjecture for hereditary rings via Ext-orthogonal pairs},
  journal = {Adv. Math.},
  volume  = {225},
  number  = {5},
  year    = {2010},
  pages   = {2341--2364}
}

\textsc{RC: University of California Davis, Dept. of Mathematics, USA.\\}
\textit{Email address:} \texttt{casals@ucdavis.edu}
\vspace{0.2cm}

\textsc{AS: University of California Davis, Dept. of Mathematics, USA.\\}
\textit{Email address:} \texttt{asimons@ucdavis.edu}
\vspace{0.2cm}

\end{document}